\documentclass[12pt,reqno]{amsart}

\usepackage[margin=1.15in]{geometry}

\usepackage{amsmath,amscd,amssymb,amsfonts,latexsym,wasysym, mathrsfs, mathtools,hhline,xcolor}
\usepackage[all, cmtip]{xy}
\usepackage{csquotes}
\usepackage{url}
\usepackage{centernot}

\definecolor{hot}{RGB}{65,105,225}

\usepackage[pagebackref=true,colorlinks=true, linkcolor=hot ,  citecolor=hot, urlcolor=hot]{hyperref}

\usepackage{ textcomp }
\usepackage{ tipa }
\usepackage{tikz}
\usetikzlibrary{arrows.meta}

\usepackage{graphicx}
\usepackage[shortlabels]{enumitem}

\newlist{kotschickconditions}{enumerate}{1}
\setlist[kotschickconditions]{
  label=(\Alph*),
  ref=(\Alph*),
  leftmargin=*,
  itemsep=0pt
}

\newcommand{\condref}[1]{%
  \hyperref[cond:#1]{\ref*{cond:#1}}%
}

\theoremstyle{plain}
\newtheorem{theorem}{Theorem}
\newtheorem{proposition}[theorem]{Proposition}

\newtheorem{lm}[theorem]{Lemma}

\newtheorem{corollary}[theorem]{Corollary}

\newtheorem{lemma}[theorem]{Lemma}
\newtheorem{thrm}[theorem]{Theorem}

\theoremstyle{definition}

\newtheorem{defn}[theorem]{Definition}

\newtheorem{remark}[theorem]{Remark}

\newtheorem{ex}[theorem]{Example}
\newtheorem*{ex*}{Example}

\newtheorem*{conjecture*}{Conjecture}
\def\be{\begin{equation}}
\def\ee{\end{equation}}

\def\bt{\begin{thrm}}
\def\et{\end{thrm}}

\def\bc{\begin{cor}}
\def\ec{\end{cor}}

\def\br{\begin{rmk}}
\def\er{\end{rmk}}

\def\bp{\begin{prop}}
\def\ep{\end{prop}}

\def\bl{\begin{lm}}
\def\el{\end{lm}}

\def\bex{\begin{ex}}
\def\eex{\end{ex}}

\def\bd{\begin{defn}}
\def\ed{\end{defn}}

\DeclareMathOperator{\Bl}{Bl}

\def\bP{\mathbb{P}}

\def\bQ{\mathbb{Q}}

\def\bZ{\mathbb{Z}}

\def \Z{\mathbb Z}
\def \R{\mathbb R}
\def \Q{\mathbb Q}

\def \C{\mathbb C}

\title[The topology of K\"ahler manifolds]{The topology of K\"ahler manifolds and \\ 1-forms without zeros}

\author{Maur\'icio Corr\^ea}
\address{Universit\`a degli Studi di Bari Aldo Moro, Dipartimento di Matematica, Via Edoardo Orabona 4, 70125 Bari, Italy.}
\email{mauricio.barros@uniba.it}

\author{J\'anos Koll\'ar}
\address{Princeton University, Department of Mathematics, Fine Hall, Washington Road, Princeton, NJ 08544-1000, USA.}
\email{kollar@math.princeton.edu}

\author{Stefan Schreieder}
\address{Leibniz Universit\"at Hannover, Institut f\"ur Algebraische Geometrie, Welfengarten 1, 30167 Hannover, Germany.}
\email{schreieder@math.uni-hannover.de}

\author{Botong Wang}
\address{Department of Mathematics, University of Wisconsin-Madison,  480 Lincoln Drive, Madison, WI 53706-1388, USA.}
\email{wang@math.wisc.edu}
\date{\today}

\hypersetup{
  pdftitle={The topology of K\"ahler manifolds and 1-forms without zeros},
  pdfauthor={Maur\'icio Corr\^ea, J\'anos Koll\'ar, Stefan Schreieder and Botong Wang},
  pdfkeywords={Topology of varieties, Kotschick conjecture, Fibration over the circle, Holomorphic one-form, Albanese morphism, Aomoto complex.} 
}
 


\begin{document}

\begin{abstract}  
We study how the topology and geometry of a compact Kähler manifold relate to the zeros of its one-forms. We determine all implications among several closely related conditions. In particular, we construct a smooth projective variety for which the Aomoto complex is exact for every nonzero holomorphic one-form and every semisimple local system, although every closed real one-form, and hence every holomorphic one-form, has a zero.
\end{abstract}

\maketitle

\section{Introduction}

In many questions concerning compact K\"ahler manifolds, including Iitaka’s conjectures, those admitting a smooth holomorphic morphism to an abelian variety arise as extremal cases that are especially difficult to handle \cite{CH02, CH07, CH11, CP17, HPS18}. 

In 2013 Kotschick asked whether
this class can be characterized in purely topological terms.
In various forms,  positive answers were given by \cite{K,S, HS,pietig} in dimensions $\leq 3$.
These led to the expectation that, for a compact K\"ahler manifold  $X$,  the following properties might be equivalent.
\begin{enumerate}
\item[(1$^{\rm hol}$)] $X$ admits a smooth, holomorphic  morphism to 
an Abelian variety of positive dimension.
\item[(2$^{\rm hol}$)] $X$ admits a holomorphic one-form without zeros.\label{cond:A}
\item[(3$^{\rm hol}$)] There is a holomorphic one-form $\omega$ on $X$ such that for every semisimple $\mathbb{C}$-local system $V$ of finite rank on $X$, the following Aomoto complex\footnote{See \cite{MR2799182} for background on Aomoto complexes and \cite{BudurWang} for the framework of differential graded Lie algebra pairs.
} is exact
\begin{equation}\label{eq:Aomoto}
H^0(X,V) \xrightarrow{\wedge\omega} H^1(X,V)  \xrightarrow{\wedge\omega} H^2(X,V)  \xrightarrow{\wedge\omega} \cdots.
\end{equation}
\item[(1$^{\rm top}$)] $X$ admits a $C^\infty$  submersion to  a positive-dimensional torus (hence to a circle).
\item[(2$^{\rm top}$)] $X$ admits a real, closed one-form without zeros.\label{cond:B}
\item[(3$^{\rm top}$)] There exists  $\omega\in H^1(X,\mathbb C) $ such that for every semisimple $\mathbb{C}$-local system $V$ of finite rank on $X$, the Aomoto complex \eqref{eq:Aomoto} is exact.
\end{enumerate}

In \cite{invisible}, it is shown that (2$^{\rm hol}$) is strictly stronger than (2$^{\rm top}$). In this note, we completely characterize all implications among these six properties. 

\begin{theorem}\label{thm:implications}
The above six properties are related as follows:
\[
(1^{\rm hol})\Longrightarrow (2^{\rm hol}) \Longrightarrow (1^{\rm top}) \Longleftrightarrow (2^{\rm top}) \Longrightarrow (3^{\rm top})\Longleftrightarrow (3^{\rm hol}).
\]
Moreover, none of the one-way implications can be reversed.
\end{theorem}

\begin{remark}
In Properties (3$^{\rm hol}$) and (3$^{\rm top}$), we require the local systems $V$ to be semisimple since non-abelian Hodge theory (see \cite{Simpson, BudurWang}) only applies to semisimple local systems. This assumption is also necessary for some of the implications. Indeed, if arbitrary local systems are allowed, then both implications $(1^{\rm top})\Rightarrow (3^{\rm hol})$ and $(1^{\rm top})\Rightarrow (3^{\rm top})$ fail. For example, let $X$ be an elliptic curve. For any $\omega\in H^1(X,\mathbb{C})$, let $V_\omega$ be the extension of $\mathbb{C}_X$ by itself corresponding to $\omega$. Then the Aomoto complex $(H^*(X, V_{\omega}), \wedge \omega)$ has a nontrivial cohomology group in degree 0. Thus, $X$ satisfies $(1^{\rm top})$ but satisfies neither $(3^{\rm top})$ nor $(3^{\rm hol})$ when arbitrary local systems are allowed.
\end{remark}

The implications $(1^{\rm hol})\Rightarrow (2^{\rm hol}) \Rightarrow (2^{\rm top})$ and $(3^{\rm hol})\Rightarrow (3^{\rm top})$ are obvious. The equivalence $(1^{\rm top}) \Leftrightarrow (2^{\rm top})$ follows from Tischler's theorem (\cite{T}). We will prove the implications $(1^{\rm top})\Rightarrow (3^{\rm top})$ and $(3^{\rm top})\Rightarrow (3^{\rm hol})$ in Section \ref{sec:resonance}.


The fact that $(2^{\rm top})$ does not imply $(2^{\rm hol})$ was established in \cite{invisible}. The following example shows that $(2^{\rm hol})$ does not imply $(1^{\rm hol})$. 
\begin{ex*}\cite{SY}
Let $X$ be the blow-up of $E_1\times E_2\times \mathbb{P}^1$ along the union of $E_1\times 0\times 0$ and $0\times E_2\times \infty$, where $E_1$, $E_2$ denote non-isogenous elliptic curves. Then, a general holomorphic one-form on $X$ has no zeros, but any holomorphic map from $X$ to a positive-dimensional abelian variety is not smooth. 
\end{ex*}

The fact that $(3^{\rm hol})$ does not imply $(2^{\rm top})$ is proved as the following theorem. 

\begin{theorem}\label{thm:Aomoto-intro}
There is a smooth, complex, projective fivefold $X$ with $b_1(X)=2$ such that, for every semisimple local system $V$ of finite rank on $X$ and  nonzero $\omega\in H^0(X,\Omega_{X}^1)$, the Aomoto complex \eqref{eq:Aomoto} is exact, but every real closed one-form on $X$ has a zero. 
\end{theorem}

The proof of Theorem \ref{thm:Aomoto-intro} relies on the following positive answer to one of the topological versions of \cite[Question 27]{KP}.  

\begin{theorem} \label{thm:KP.rightway.thm}
Let $\Sigma$ be a (possibly non-compact) Riemann surface with $\pi_1(\Sigma)\neq 0$. 
Let $g\colon X\to \Sigma$ be a proper morphism between complex spaces. 
Assume that $g$ has simply connected fibers. 
Then the following are equivalent.
   \begin{enumerate}[(i)]
    \item $g$ is a $\Z$-homology fiber bundle. \label{item:KP.rightway.thm.1}
    \item $g$ is a homotopy fiber bundle.\label{item:KP.rightway.thm.2}
    \item  The universal cover $\widetilde X$ is homotopy equivalent to a finite CW complex. \label{item:KP.rightway.thm.3}
    \end{enumerate}
    If, in addition,   $\dim X\geq 4$ and $X$ is smooth, then these are also equivalent to
    \begin{enumerate}[(i)] \setcounter{enumi}{3}
    \item \label{item:KP.rightway.thm.4} There exists a proper $C^\infty$-fiber bundle $g'\colon X\to\Sigma$ which is homotopic to $g$ and agrees with $g$ outside a small neighborhood of the singular fibers of $g$. 
    \end{enumerate}
\end{theorem}

\begin{remark}
The following was introduced in \cite{S}  to study Kotschick's question: 
\begin{enumerate}
\item[$(3^{\rm \acute{e}tale})$] There exists $\omega\in H^1(X, \mathbb{C})$ such that, for every finite \'etale cover $\pi\colon X'\to X$, the Aomoto complex
$$
H^0(X',\mathbb C) \xrightarrow{\wedge\pi^*\omega} H^1(X',\mathbb C)  \xrightarrow{\wedge\pi^*\omega} H^2(X',\mathbb C)  \xrightarrow{\wedge\pi^*\omega} \cdots
$$
is exact.
\end{enumerate}
It was proved in \cite{S} that $(1^{\rm hol})\Rightarrow (3^{\rm \acute{e}tale})$. Moreover, \cite{HS,pietig} proved that, when $\dim X\leq 3$, $(3^{\rm \acute{e}tale})$ implies $(2^{\rm top})$. Together, these results give a positive answer to Kotschick's question for K\"ahler manifolds of dimension at most three.

For any finite \'etale cover $\pi\colon X'\to X$, the local system $\pi_*(\mathbb{C}_{X'})$ is semisimple. Thus, by the isomorphism $H^i(X',\mathbb C)\cong H^i(X, \pi_*(\mathbb{C}_{X'}))$, Property $(3^{\rm hol})$ implies $(3^{\rm \acute{e}tale})$. Since Theorem \ref{thm:Aomoto-intro} shows that $(3^{\rm hol})$ does not imply $(2^{\rm top})$, it follows that $(3^{\rm \acute{e}tale})$ does not imply $(2^{\rm top})$ in higher dimensions. 
\end{remark}

The paper is organized as follows. Section~\ref{sec:resonance} studies resonance varieties and uses non-abelian Hodge theory to complete the implications in Theorem~\ref{thm:implications}. Section~\ref{sec:KP-topological} proves the topological equisingularity theorem (Theorem~\ref{thm:KP.rightway.thm}). Section~\ref{sec:Aomoto} constructs the counterexample in Theorem~\ref{thm:Aomoto-intro}, and analyzes its foliated geometry.

\section{Resonance varieties}\label{sec:resonance}
Given a topological space $X$ and a local system $V$ on $X$, we define the associated \emph{resonance variety} $R(X, V)$ to be the set of classes $\alpha\in H^1(X, \mathbb{C})$ for which the complex
\[
H^0(X, V)\xrightarrow{\wedge \alpha} H^1(X, V)\xrightarrow{\wedge \alpha} H^2(X, V)\xrightarrow{\wedge \alpha} \cdots
\]
is not exact. 
In the literature, resonance varieties usually refer to the loci $R^i_k(X, V)$ consisting of classes $\alpha\in H^1(X, \mathbb{C})$ for which the $i$-th cohomology of the above complex has dimension at least $k$. They are topological invariants of $(X, V)$ that connect geometry and combinatorics (see e.g. \cite{Falk, DimcaPapadima}). Our $R(X, V)$ is the union of $R^i_1(X, V)$ for all $i$. 

For an arbitrary topological space $X$ and local system $V$, $R(X, V)$ is a cone but need not have any further special geometric structure. However, when $X$ is a compact K\"ahler manifold and $V$ is semisimple, non-abelian Hodge theory imposes strong constraints on the geometry of $R(X, V)$. 

\begin{theorem}\label{sub-Hodge}
    Assume that $X$ is a compact K\"ahler manifold and $V$ is a semisimple local system of finite rank on $X$. Then, the resonance variety $R(X, V)$ is a finite union of sub-Hodge structures in $H^1(X, \mathbb{C})$. 
\end{theorem}
\begin{proof}
    The resonance variety $R(X, V)$ is closely related to the cohomology support locus, whose definition we now recall. Let $\mathrm{M}_B(X, 1)$ denote the moduli space of rank-one local systems on $X$, which is a commutative algebraic group. We define the \emph{cohomology support locus} of $V$ by
    \[
    \Sigma(X, V)\coloneqq \{L\in \mathrm{M}_B(X, 1)\mid H^i(X, V\otimes L)\neq 0 \text{ for some $i$}\}.
    \]
    By \cite[Theorem 7.3]{Schnell}, $\Sigma(X, V)$ is a finite union of linear subvarieties.\footnote{In \cite{Schnell}, the result assumes that $X$ is projective. However, the argument also generalizes to the compact K\"ahler case (see also \cite{Arapura2}). The idea is that $\mathrm{M}_B(X, 1)$ has a natural hyperK\"ahler structure and $\Sigma(X, V)$ is a quaternion subvariety, which implies that each of its irreducible components is a linear subvariety.} In other words, up to a translate, every irreducible component of $\Sigma(X, V)$ is of the form 
    \[
    \operatorname{Image}\left(\varphi^*\colon \mathrm{M}_B(A, 1)\to \mathrm{M}_B(X, 1)\right)
    \]
    where $\varphi\colon X\to A$ is a map to a compact complex torus.\footnote{When $X$ is a smooth projective variety, $A$ can be chosen to be an abelian variety.}

    Notice that the Lie algebra of the algebraic group $\mathrm{M}_B(X, 1)$ is naturally isomorphic to $H^1(X, \mathbb{C})$. Hence there is an exponential map $\exp\colon  H^1(X, \mathbb{C})\to \mathrm{M}_B(X, 1)$. By \cite[Theorem 1.4]{BudurWang}, the exponential map induces an isomorphism of analytic germs
    \begin{equation}\label{eq:germ}
    R(X, V)_{(0)}\cong \Sigma(X, V)_{(1)}.
    \end{equation}
    Since $R(X, V)$ is a cone, the linearity of $\Sigma(X, V)$ in the previous paragraph implies that every irreducible component of $R(X, V)$ is a sub-Hodge structure in $H^1(X, \mathbb{C})$. 
\end{proof}

\begin{corollary}
  For any compact K\"ahler manifold $X$, \textup{$(1^{\rm top}) \Rightarrow (3^{\rm top})$.}
\end{corollary}
\begin{proof}
    Assume that $f\colon X\to S^1$ is a smooth submersion. Then, for any local system $V$ on $X$, $R^if_*(V)$ is a local system for every $i$. Choose a general local system $L_{S^1}$ on $S^1$ such that 
    \[
    H^0(S^1, R^if_*(V)\otimes L_{S^1})=H^1(S^1, R^if_*(V)\otimes L_{S^1})=0
    \]
    for every $i$. 
    Then, by the Leray spectral sequence, 
    \[
    H^i(X, V\otimes f^*(L_{S^1}))=0
    \]
    for every $i$. 
    
    Notice that $f^*L_{S_1}$ corresponds to a point of $\mathrm{M}_B(X, 1)$ in the same connected component of the trivial local system. Therefore, $\Sigma(X, V)$ does not contain the whole connected component of $\mathrm{M}_B(X, 1)$ containing the trivial local system. Now, assume further that $V$ is semisimple. Then by \eqref{eq:germ}, $R(X, V)$ is a proper subset of $H^1(X, \mathbb{C})$. 

    By Theorem \ref{sub-Hodge}, for any semisimple local system $V$ on $X$, $R(X, V)$ is a finite union of proper sub-Hodge structures in $H^1(X, \mathbb{C})$. Since there are at most countably many sub-Hodge structures in $H^1(X, \mathbb{C})$, a very general $\omega\in H^1(X, \mathbb{C})$ satisfies the condition of $(3^{\rm top})$ for any semisimple local system $V$ on $X$. 
\end{proof}

\begin{corollary}
For any compact K\"ahler manifold $X$, \textup{$(3^{\rm top}) \Rightarrow (3^{\rm hol})$.}
\end{corollary}
\begin{proof}
    Assume that $X$ satisfies property $(3^{\rm top})$. Then, for any semisimple local system $V$ on $X$, the resonance variety $R(X, V)$ is a proper subset of $H^1(X, \mathbb{C})$. By Theorem~\ref{sub-Hodge}, $R(X, V)$ is a finite union of proper sub-Hodge structures. For each such sub-Hodge structure $S$, the intersection $S\cap H^{1,0}(X)$ is a proper linear subspace of $H^{1,0}(X)$. Since there are at most countably many sub-Hodge structures in $H^1(X, \mathbb{C})$, a very general element $\omega\in H^{1,0}(X)$ is not contained in any proper sub-Hodge structure of $H^1(X, \mathbb{C})$. Such an element $\omega$ satisfies the condition of $(3^{\rm hol})$ for any semisimple local system $V$ on $X$. 
\end{proof}

\section{A topological equisingularity theorem} \label{sec:KP-topological}
In this section we prove Theorem \ref{thm:KP.rightway.thm}.
 We will need the following proposition.

\begin{proposition} \label{prop:D-finite}
Let $X$ be a complex space and let $f\colon X\to \Sigma$ be a proper morphism to a Riemann surface $\Sigma$ such that $b_1(\Sigma)<\infty$.
Then the groups $H^i(X,\Z)$ are finitely generated if and only if there is a finite subset $D\subseteq \Sigma$ such that $f$ is a $\mathbb{Z}$-homology fiber bundle over $\Sigma\setminus D$.
\end{proposition}

\begin{proof}
There is a smallest discrete subset $D\subseteq \Sigma$ such that the sheaves $R^if_*\Z_X$ are local systems on $\Sigma\setminus D$.
If $D$ is finite, then the groups
$$
H^j\bigl(\Sigma\setminus D,R^if_*\Z_X\bigr)
$$
are finitely generated, and hence so are the groups
$$
H^j\bigl(X\setminus f^{-1}(D),\mathbb{Z}\bigr)
$$
by the Leray spectral sequence.
The Mayer--Vietoris sequence now gives finite generation of the groups $H^j(X,\mathbb{Z})$.

If $D$ is infinite, then $\Sigma$ is noncompact, hence Stein. 
Artin's vanishing theorem for constructible sheaves on Stein manifolds therefore gives
$$
H^j\bigl(\Sigma,R^if_*\Z_X\bigr)=0 \qquad\text{for }j>1,
$$
see \cite[Theorem~10.3.8]{KS90}. 
Thus all differentials in the Leray spectral sequence vanish.
Therefore, the groups $H^i(X,\Z)$ are finitely generated if and only if each group
$$
H^j\bigl(\Sigma,R^if_*\Z_X\bigr)
$$
is finitely generated.
The proposition therefore follows from Lemma \ref{lem:D-finite} below.
\end{proof}

\begin{lemma} \label{lem:D-finite}
Let $\Sigma$ be a connected orientable topological surface with $b_1(\Sigma)<\infty$.
Let ${\mathcal F}$ be a sheaf of abelian groups on $\Sigma$ whose stalks are finitely generated. 
Let
$$
\Sigma^\circ=\{x\in \Sigma \mid {\mathcal F}\text{ is a local system on a neighborhood of }x\}
$$
and put $D=\Sigma\setminus \Sigma^\circ$.  Assume that $D$ is discrete.  Then the groups
$H^i(\Sigma,\mathcal{F})$ are finitely generated for every $i$ if and only if $D$ is finite. 
\end{lemma}

\begin{proof}
If $D$ is finite, then the claim is clear.
Conversely, assume that $D$ is infinite.
Let ${\mathcal F}_{sky}\subset {\mathcal F}$ be the maximal skyscraper subsheaf. 
If ${\mathcal F}_{sky}$ has infinitely many nonzero stalks, $H^0(\Sigma,\mathcal{F})$ is not finitely generated.
We may thus assume that this is not the case.
Then we replace ${\mathcal F}$ by ${\mathcal F}/{\mathcal F}_{sky}$ and assume that ${\mathcal F}$ does not contain skyscraper subsheaves.



Note that $D$ is closed and discrete, hence is countable and has no accumulation points in $\Sigma$.
In particular, we can write
$$
D=\{0_i\mid i=1,2,\ldots\}.
$$
Since $\Sigma$ is orientable and $b_1(\Sigma)<\infty$, it is given by attaching to a compact orientable surface (called compact core) with boundary an infinite annulus to each of the finitely many boundary components.
We can therefore choose a closed embedded tree $\Gamma\subset\Sigma$ which contains $D$ and is obtained from a finite tree contained in the compact core by attaching finitely many proper rays, one for each of the attached ends.  

Let $W_1$ be a sufficiently small open thickening of $\Gamma$, chosen
so that $\Gamma$ is a deformation retract of $W_1$, and put
$
W_2\coloneq\Sigma\setminus\Gamma.
$
Then $W_1$ is connected and contractible, while $W_2$ and $W_1\cap W_2$ have the homotopy type of finite CW complexes and do not contain any points of $D$.
The Mayer--Vietoris sequence therefore reduces us to the case where $\Sigma=W_1$ is a connected contractible surface without boundary.

We then choose an open covering
$\Sigma=V\cup\bigsqcup_{i=1}^\infty U_i$, where the $U_i$ are pairwise
disjoint open discs containing $0_i$, $V\cap U_i$ is nonempty and
contractible, and $V$ is connected and disjoint from $D$.
(For instance, after identifying $\Sigma$ with $\R^2$, we may take
$0_i=(i,0)$, $U_i=(i-\tfrac14,i+\tfrac14)\times\R$,
$R_i=\{i\}\times(-\infty,0]$, and
$V=\R^2\setminus\bigcup_{i\geq1}R_i$.
Then $\R^2=V\cup\bigcup_{i\geq1}U_i$, and each $V\cap U_i$ is nonempty
and contractible.)
By the definition of $\Sigma^\circ$ and $D$,
the sheaf ${\mathcal F}$ is locally constant outside of $D$, but ${\mathcal F}$ is not locally constant in the neighborhood of any of the points $0_i$.



\medskip

\textbf{Claim.} The map $H^0(U_i,{\mathcal F})\to H^0(V\cap U_i,{\mathcal F})$ is not surjective.
 
\begin{proof} 
Set
$$
j\colon U_i\setminus\{0_i\}\hookrightarrow U_i \qquad\text{and}\qquad \mathcal L\coloneq \mathcal F|_{U_i\setminus\{0_i\}}.
$$
There is a natural morphism
\begin{align} \label{eq:claim:map:1}
\mathcal F|_{U_i}\longrightarrow j_\ast \mathcal L.
\end{align}
Its kernel is supported at $0_i$, hence is a skyscraper subsheaf of $\mathcal F$. 
By our reduction, this kernel is zero.

Suppose, for contradiction, that
\begin{align} \label{eq:claim:map:2}
H^0(U_i,\mathcal F)\longrightarrow H^0(V\cap U_i,\mathcal F)
\end{align}
is surjective. 
Since $V\cap U_i$ is nonempty and contractible and $\mathcal F$ is locally constant there, its space of sections identifies with any stalk of $\mathcal L$ on $V\cap U_i$. 
Surjectivity of \eqref{eq:claim:map:2} thus implies that $\mathcal L$ is a trivial local system on $U_i\setminus\{0_i\}$.
But then $j_\ast \mathcal L$ is a trivial local system on $U_i$ and the injective map \eqref{eq:claim:map:1} is surjective by surjectivity of  \eqref{eq:claim:map:2}.
Hence, \eqref{eq:claim:map:1} is an isomorphism, which contradicts our assumptions.  
\end{proof}

Let $U=\bigsqcup_{i\geq 1} U_i$ and consider the Mayer--Vietoris sequence
$$
H^0(\Sigma,{\mathcal F}) \to \prod_{i=1}^{\infty} H^0(U_i,{\mathcal F}) \oplus H^0(V,{\mathcal F})\to \prod_{i=1}^{\infty} H^0(U_i\cap V,{\mathcal F})\to H^1(\Sigma,{\mathcal F}) .
$$
By the above claim, the cokernel of $\prod_{i=1}^{\infty} H^0(U_i,{\mathcal F}) \to \prod_{i=1}^{\infty} H^0(U_i\cap V,{\mathcal F})$ is not finitely generated.
Since ${\mathcal F}$ has no skyscraper subsheaf and the stalks of ${\mathcal F}$ are finitely generated,  $H^0(V,{\mathcal F})$ is finitely generated.
Hence, the above exact sequence shows that $H^1(\Sigma,{\mathcal F})$ is not finitely generated.
This concludes the proof of the lemma.
\end{proof}

\begin{proof}[Proof of Theorem \ref{thm:KP.rightway.thm}]
Item \ref{item:KP.rightway.thm.1} implies \ref{item:KP.rightway.thm.2} by \cite[Lemma~5(i)]{invisible}.
Since $\pi_1(\Sigma)\neq 0$, the universal cover of $\Sigma$ is contractible, and so \ref{item:KP.rightway.thm.2} implies \ref{item:KP.rightway.thm.3} by applying \cite[Lemma~5(ii)]{invisible} to the pullback of $g$ to the universal cover of $\Sigma$.

By a theorem of Smale \cite[Main Theorem]{Smale}, $g\colon X\to \Sigma$ induces an isomorphism on fundamental groups, because it has simply connected fibers.
Hence, $\widetilde X=X\times_{\Sigma}\widetilde \Sigma$, where $\widetilde X$ and $\widetilde \Sigma$ denote the universal covers of $X$ and $\Sigma$, respectively.
Since $\pi_1(\Sigma)\neq 0$, the classification of Riemann surfaces implies that $\widetilde \Sigma\to \Sigma$ has infinite degree.
Moreover, $b_1(\widetilde \Sigma)=0$. 
Therefore, item \ref{item:KP.rightway.thm.3} implies \ref{item:KP.rightway.thm.1} by Proposition \ref{prop:D-finite}, applied to the induced map $\widetilde X\to \widetilde \Sigma$.

Assume now that $X$ is smooth. 
By \cite[Proposition~14]{invisible}, item \ref{item:KP.rightway.thm.1} implies item \ref{item:KP.rightway.thm.4} if $\dim X\geq 4$. 
Conversely, if \ref{item:KP.rightway.thm.4} holds, then the universal cover $\widetilde X$ of $X$ is  a $C^\infty$-fiber bundle over the universal cover $\widetilde \Sigma$ of $\Sigma$.
Since $\Sigma$ is not simply connected, $\widetilde \Sigma$ is either $\C$ or the disc $\Delta$.
It follows that $\widetilde X$ is homotopy equivalent to a fiber of $\widetilde X\to \widetilde \Sigma$ and so \ref{item:KP.rightway.thm.3} holds. 
\end{proof}

\begin{remark}
When $\Sigma$ is compact, then the equivalence of \ref{item:KP.rightway.thm.1}, \ref{item:KP.rightway.thm.2} and \ref{item:KP.rightway.thm.3} in Theorem~\ref{thm:KP.rightway.thm} can be deduced from \cite[Propositions 5.5 and 5.6, and Lemma 5.9]{LMW}; see also \cite[Corollary 1.6]{LMW} for the case of elliptic curves.
\end{remark}
 
\section{One-forms without zeros and exactness of the Aomoto complex} \label{sec:Aomoto}

In this section we prove Theorem \ref{thm:Aomoto-intro}, stated in the introduction.

\subsection{Preliminaries}
We will use the following blow-up formula in the derived category.
For the analogous statement with $\ell$-adic coefficients on the
\'etale site, see \cite[Exp.~VII, Thm.~2.2.1 and Thm.~8.1]{SGA5}.
See also \cite[Thm.~7.31]{voisin} for the resulting decomposition of
integral cohomology groups.

\begin{lemma} 
\label{lem:derived-blow-up-formula}
Let $M$ be a complex manifold, let $\iota\colon Z\hookrightarrow M$ be a
closed complex submanifold of codimension $c$, and let
\[
\pi\colon M'=\Bl_ZM\longrightarrow M
\]
be the blow-up of $M$ along $Z$. Then there is an isomorphism
\begin{equation}\label{eq:derived-blow-up-formula}
R\pi_{*}\Z_{M'}
\simeq
\Z_M\oplus\bigoplus_{r=1}^{c-1}\iota_*\Z_Z[-2r]
\end{equation}
in the derived category of sheaves of abelian groups on $M$.
\end{lemma}

\begin{proof}
Let $j\colon E\hookrightarrow M'$ be the exceptional divisor and let
$p\colon E=\bP(N_{Z/M})\to Z$ be its natural projection. The morphism in
\eqref{eq:derived-blow-up-formula} is obtained from the pullback
$ \Z_M\to R\pi_*\Z_{M'}$
together with, for $1\leq r\leq c-1$, the morphisms
$\iota_*\Z_Z[-2r]\to R\pi_*\Z_{M'}
$
defined by pulling back to $E$, taking cup product with
$c_1(\mathcal O_E(1))^{r-1}$, and applying the Gysin morphism for
$j\colon E\hookrightarrow M'$.

It suffices to check that the resulting morphism is an isomorphism on
stalks. 
Away from $Z$, this is clear. 
At a point of $Z$, the assertion is precisely the projective-bundle formula for the fiber $\bP^{c-1}$.
This proves \eqref{eq:derived-blow-up-formula}.
\end{proof}

We will also need the following generic projection lemma.

\begin{lemma}\label{lem:relative-generic-projection}
Let $B$ be a smooth curve and 
$
q\colon Y\to B
$ 
 a nonconstant morphism from a smooth projective variety of dimension $m$.
Then there is a closed embedding
$$
\iota\colon Y\hookrightarrow \mathbb P^{2m}\times B\qquad \text{such that}\quad \operatorname{pr}_2\circ\iota=q.
$$ 
\end{lemma}
\begin{proof}
Choose a closed embedding $j\colon Y\hookrightarrow\mathbb P^n$, with
$n$ sufficiently large.  
Let $\Sigma\subset\mathbb P^n$ be the closure of
the union of all lines $\overline{j(x)j(y)}$, where $x\neq y$ and
$q(x)=q(y)$, together with all tangent lines to $j(Y)$ corresponding to
nonzero vectors in $\ker(dq_y)$ for some $y\in Y$. 
One checks that $\dim \Sigma\leq 2m$.
Hence, if $p\colon \bP^n\dashrightarrow \bP^{2m}$ is a general linear projection, then $\iota\coloneq (p\circ j,q)$ has the properties we want. 
\end{proof}

\subsection{A rational cohomology torus}\label{sec2} 
We recall the Iitaka torus tower from \cite{DJL}.  

Let $E$ and $F$ be two elliptic curves, and let 
$$
p_C \colon C\longrightarrow E
$$ 
be a branched double cover of $E$. 
Let $\sigma_1\colon F\to F$ be an order-two automorphism defined by translating by a nonzero 2-torsion point, and denote the deck involution of $p_C$ by $\sigma_2$. 
Let $\sigma=(\sigma_1, \sigma_2)$ be the induced diagonal action on $F\times C$, and define
$$
S=(F\times C)/\langle \sigma\rangle.
$$
Since $\sigma$ is free, $S$ is a smooth projective surface. Let 
\begin{align} \label{eq:p-Aomoto}
p\colon S\longrightarrow C/\langle \sigma_2\rangle=E
\end{align}
be the natural projection. 
Then, away from the branch locus of $p_C$, the fibers of $p$ are isomorphic to $F$. 
Over the branch locus of $p_C$, $p$ has double fibers, and these fibers (with induced reduced scheme structure) are isomorphic to $F/\langle \sigma_1\rangle$. 
The important feature of $S$ is that $R^ip_* \bQ_S$ is a trivial local system on $E$ of rank $1$, $2$, $1$, when $i=0, 1, 2$, respectively, and for $i>2$, $R^ip_\ast \bQ_S=0$, cf.~Lemma \ref{lemma:pushforward} below. 
This implies that the Albanese map of $S$ induces an isomorphism of rational cohomology groups. In particular, the rational cohomology ring of $S$ is isomorphic to the rational cohomology ring of a 2-dimensional complex torus.
However, since $p$ has double fibers at the branch locus of $p_C$, $R^2p_* \bZ_S$ is not a local system near those points. 

Geometrically, $p$ is a locally trivial fibration with fiber $F$ away from the branch locus of $p_C$. 
Since locally in $S$, the map $p$ factors through the ramified cover $p_C\colon C\to E$, $p$ fails to be a submersion precisely over the branch locus of $p_C$. 
 
\subsection{Construction of the example}\label{sec:example}
Let $N\ge 3$ be such that there is a closed embedding $\iota \colon S\to \bP^N\times E$ which lifts $p\colon S\to E$ from \eqref{eq:p-Aomoto}.
Via this map we consider $S$ as a subvariety of $\bP^N\times E$. 
Let 
\begin{align} \label{def:X-1}
X\coloneq \Bl_S(\bP^N\times E),
\end{align}
which is a smooth projective variety of dimension $N+1$.
We further consider the projection
$$
h\colon X\longrightarrow E
$$
which agrees with the Albanese map of $X$ and induces an isomorphism on fundamental groups.  

In the remainder of this section, $X$ will always denote the blow-up in \eqref{def:X-1} and we will show that it satisfies the conclusions in Theorem \ref{thm:Aomoto-intro}.

\subsection{Cohomology of the Aomoto complexes}\label{sec4}
In this section, we prove the following proposition, which shows that $X$ satisfies an even stronger condition than Properties $(3^{\rm hol})$ and $(3^{\rm top})$. 

\begin{proposition}\label{prop:vanishing}
Let $V$ be any semisimple local system on $X$ and assume $\omega\in H^1(X, \mathbb{C})$ is nonzero. 
Then the Aomoto complex
$$
(H^*(X, V), \wedge\omega)\coloneqq H^0({X}, V)\xrightarrow{\wedge \omega} H^1({X}, V)\xrightarrow{\wedge \omega}H^2({X}, V)\xrightarrow{\wedge \omega}\cdots 
$$
is exact. 
\end{proposition}

We begin with the following lemma.

\begin{lemma}\label{lemma:pushforward}
The map $p\colon S\to E$ satisfies:
\begin{equation}\label{eq:trivial LS}
Rp_* \bQ_S \cong \bQ_E\oplus \bQ_E[-1]^{\oplus 2}\oplus \bQ_E[-2].
\end{equation}
\end{lemma}
\begin{proof}
    First, we prove that $p\colon S\to E$ is a $\Q$-homology fiber bundle. 
    In fact, away from the branch locus of $p_C$, $p$ is an $F$-fiber bundle. 
    Let $x\in E$ be a branch point of $p_C$, and let $x'\in E$ be a point different from but very close to $x$. 
    Then there is a specialization map between the fibers of $p$, $\mathrm{sp}\colon S_{x'}\to S_{x}$, well-defined up to homotopy. 
    This map can be chosen to be the finite \'etale double cover $F\to F/\langle \sigma_1 \rangle$. 
    Hence, it induces an isomorphism on rational (co-)homology.
    Therefore, $p\colon S\to E$ is a $\Q$-homology fiber bundle. 
    In particular, $R^ip_* \bQ_S$ is a local system for every $i$. 

    Away from the branch locus of $p_C$, $p$ is an $F$-fiber bundle, and the monodromy action is given by translation. 
    Thus, the restriction of $R^ip_* \bQ_S$ to $E^\circ\coloneq E\setminus \operatorname{Br}(p_C)$ is the trivial local system. 
    Since the inclusion $E^\circ\to E$ induces a surjective map on fundamental groups, $R^ip_* \bQ_S$ is a trivial local system for all $i$. 
    Thus, the isomorphism \eqref{eq:trivial LS} follows from the decomposition theorem \cite{BBD} and the fact that a general fiber of $p$ is an elliptic curve. 
\end{proof}

\begin{corollary}\label{cor:sum} 
There are $a_i\in \bZ_{\geq 0}$ with  
\[
R{h}_* \bQ_{X} \cong \bigoplus_{0\leq i\leq 2N}\bQ_{E}[-i]^{\oplus a_i} .
\] 
\end{corollary}
\begin{proof}
Recall that ${h}\colon {X}\to {E}$ is equal to the composition of the blow-up map ${X}\to \bP^N\times {E}$ and the projection $\bP^N\times {E}\to {E}$. 
Taking the derived pushforward of $\bQ_{X}$ to $\bP^N\times {E}$, we obtain a direct sum of the constant sheaf $\bQ_{\bP^N\times {E}}$ with a direct sum of shifted trivial local systems on ${S}$, cf.~\eqref{eq:derived-blow-up-formula}. 
Note further that the derived pushforwards of $\bQ_{\bP^N\times {E}}$ and $\bQ_{S}$ to ${E}$ are direct sums of shifted trivial local systems on ${E}$. 
Hence, $R{h}_*(\bQ_{X})$ is isomorphic to a direct sum of shifted trivial local systems on ${E}$.
The precise version in the corollary follows because $h$ has relative dimension $N$. 
\end{proof}

\begin{proof}[Proof of Proposition \ref{prop:vanishing}]
Notice that $h: X\to E$ induces an isomorphism on the fundamental groups. Thus, $\pi_1(X)\cong \mathbb{Z}^2$ is abelian, and it suffices to consider rank-one local system $V=h^*V_E$. By Corollary~\ref{cor:sum} and the projection formula, 
\[
Rh_*(V)\cong \bigoplus_{0\leq j\leq 2N}V_{E}[-j]^{\oplus a_j}.
\]
If $V_E$ is nontrivial, then $H^*(E, V_E)=0$. Hence, $H^{*}(X, V)=0$, and the assertion follows. 

If $V_E$ is trivial, then the above decomposition gives
\[
H^*(X, V)\cong \bigoplus_{0\leq j\leq 2N}H^{*}(E, \mathbb{C})[-j]^{\oplus a_j}
\]
as graded $H^{*}(E, \mathbb{C})$-modules. In particular, $H^*(X, V)$ is a free graded $H^*(E, \mathbb{C})$-module. 

Since $h^*: H^1(E, \mathbb{C})\to H^1(X, \mathbb{C})$ is an isomorphism, we can write $\omega=h^*\alpha$ for some nonzero $\alpha\in H^1(E, \mathbb{C})$. Since the complex $(H^*(E, \mathbb{C}), \wedge \alpha)$ is exact, the complex $(H^*(X, V), \wedge \omega)$ is also exact. 
\end{proof}

\subsection{Proof of Theorem \ref{thm:Aomoto-intro}} 
\begin{proof}[Proof of Theorem \ref{thm:Aomoto-intro}]
Let $X$ be as in Section \ref{sec:example}.
By Lemma \ref{lem:relative-generic-projection}, we can choose $N=4$ and so $\dim X=5$.
Let $V$ be any semisimple local system on $X$. 
By Proposition \ref{prop:vanishing}, for any nonzero $\omega\in H^0(X,\Omega^1_{X})$, the Aomoto complex $(H^\ast(X,V),\wedge \omega)$ is exact.
It remains to show that any real closed one-form on $X$ has a zero.
Assume, for contradiction, that $X$ admits a real closed one-form without zeros.
Then $X$ fibers smoothly over $S^1$, see \cite{T}.  
We can pick such a $C^\infty$-fibration $X\to S^1$ such that $\pi_1(X)\to \pi_1(S^1)$ is surjective. 
Let $\widetilde X\to X$ be the infinite cyclic cover corresponding to the kernel of $\pi_1(X)\to \pi_1(S^1)$.
Then $\widetilde X$ is homotopy equivalent to a fiber of $X\to S^1$, hence has the homotopy type of a finite CW complex.

Recall that $h\colon X\to E$ induces an isomorphism on $\pi_1$ and let $\widetilde E\to E$ be the infinite cyclic covering, induced by the kernel of $\pi_1(E)\cong \pi_1(X)\to \pi_1(S^1)$.
Then $h$ induces a natural map $\tilde h\colon \widetilde X\to \widetilde E$. 
This map is not a $\Z$-homology fiber bundle. 
Indeed, $R^{4}\widetilde h_\ast \Z$ fails to be locally constant at every
point of $\widetilde E$ lying above the branch locus of $p_C\colon C\to E$, and this is an infinite set. 
Proposition \ref{prop:D-finite} therefore implies that $H^{i}(\widetilde X,\Z)$ is not finitely generated for some $i$.
In particular, $\widetilde X$ does not have the homotopy type of a
finite CW complex. 
This contradiction completes the proof of the theorem.
%
\end{proof}


\subsection{Foliated geometry and dynamics of the example}\label{sec:foliated-geometry}

We now describe the foliated geometry underlying the example of Section \ref{sec:example}. Recall that $h\colon X\to E$ is the Albanese morphism. Every nonzero holomorphic one-form on $X$ is of the form $\omega=h^*\eta$, where $\eta$ is a nonzero holomorphic one-form on $E$. On $X\setminus Z(\omega)$, let $\mathcal F_\omega$ denote the holomorphic foliation defined by $\ker\omega$.

The period homomorphism $\operatorname{Per}_\omega\colon\pi_1(X)\to\C$, given by $\gamma\mapsto\int_\gamma\omega$, determines the transverse translation structure of $\mathcal F_\omega$. Its exponentials $\chi_t(\gamma)=\exp(t\int_\gamma\omega)$, $t\in\C$, define rank-one local systems in the direction $\omega$. For general $t$, these characters have infinite image, and hence contain information beyond that detected by finite \'etale covers.

Since $\pi_1(X)\simeq\Z^2$, every semisimple complex local system on $X$ is a direct sum of rank-one local systems. Thus Proposition~\ref{prop:vanishing} shows that the Aomoto complex is exact in every nonzero direction for all semisimple local systems, including the characters arising from the transverse periods of $\omega$.

For $c\in\C^*$, set $\alpha_c=\operatorname{Re}(c\omega)$ and $\beta_c=\operatorname{Re}(c\eta)$. We denote by
$$
\Gamma_c=\left\{\int_\gamma\beta_c\mid \gamma\in H_1(E,\Z)\right\}\subset\R
$$
the group of periods of $\beta_c$.

\begin{proposition}\label{prop:foliated-geometry}
With the notation above, the following statements hold.
\begin{enumerate}[(i)]
\item Over the regular values of $h$, the fibers of $h$ are compact leaves of $\mathcal F_\omega$.
\item Every regular leaf of $\ker\alpha_c$ is an immersed Levi-flat hypersurface \cite{BrunellaLevi}, and its Levi foliation is the restriction of $\mathcal F_\omega$.
\item If $\Gamma_c$ is discrete, then, after rescaling, there is a submersion
$
q_c\colon E\longrightarrow S^1
$
such that $q_c^*d\vartheta=\beta_c$. For $H_c=q_c\circ h$, every regular fiber of $H_c$ is a compact Levi-flat leaf of $\ker\alpha_c$, fibered over a circle by compact leaves of $\mathcal F_\omega$.
\item If $\Gamma_c$ is dense in $\R$ and $\ell$ is a leaf of $\ker\beta_c$ which does not contain a critical value of $h$, then $h^{-1}(\ell)$ is a leaf of $\ker\alpha_c$ and is dense in $X$. In particular, this holds for all but finitely many leaves of $\ker\beta_c$.
\end{enumerate}
Moreover,
$
\operatorname{Sing}(\ker\alpha_c)=Z(\omega)=\operatorname{Sing}(\mathcal F_\omega)=\operatorname{Crit}(h).
$
\end{proposition}

\begin{proof}
Since $\eta$ is nowhere vanishing, $Z(\omega)$ is the critical locus of $h$ and $\ker\omega=\ker dh$ away from $Z(\omega)$. Corollary~\ref{cor:sum} also gives
$
H^0(h^{-1}(t),\bQ)=\bQ
$
for every $t\in E$. Thus the fibers of $h$ are connected, and (i) follows.

Since
$
\ker\omega=\ker\alpha_c\cap J\ker\alpha_c,
$
the complex tangent bundle of every regular leaf of $\ker\alpha_c$ is $\mathcal F_\omega$. This proves (ii).

The form $\beta_c$ defines a linear foliation on the real two-torus underlying $E$. If $\Gamma_c$ is discrete, a nonzero multiple of $\beta_c$ has primitive integral periods. Integration gives a submersion $q_c\colon E\to S^1$ with $q_c^*d\vartheta=\beta_c$, as in Tischler's theorem \cite{T}. Since $\alpha_c=h^*\beta_c$, we have
$
H_c^*d\vartheta=\alpha_c.
$
This proves (iii).

If $\Gamma_c$ is dense in $\R$, every leaf of $\ker\beta_c$ is dense in $E$ \cite{CandelConlon}. Let $B\subset E$ be the finite set of critical values of $h$ and let $\ell$ be a leaf disjoint from $B$. Since $h$ is a smooth fibration with connected fibers over $E\setminus B$, $h^{-1}(\ell)$ is a leaf of $\ker\alpha_c$. Let $U\subset E\setminus B$ be any sufficiently small disc over which $h$ is a product. Since $\ell\cap U$ is dense in $U$, the closure of $h^{-1}(\ell)$ contains $h^{-1}(U)$.
Since this holds for every such $U$, the closure contains $h^{-1}(E\setminus B)$, and hence $h^{-1}(\ell)$ is dense in $X$. Since every point of $B$ lies on one leaf of $\ker\beta_c$, this proves (iv).

Finally, $\operatorname{Re}(c\omega)$ vanishes at a point if and only if $\omega$ does. Together with $Z(\omega)=\operatorname{Crit}(h)$, this proves the last assertion.
\end{proof}

\begin{remark}
Propositions~\ref{prop:vanishing} and~\ref{prop:foliated-geometry} show that the semisimple Aomoto condition does not detect the singularities or the transverse dynamics of the foliation defined by a holomorphic one-form. Indeed, for every semisimple local system $V$ on $X$ and every nonzero $\omega\in H^1(X,\C)$, the corresponding Aomoto complex is exact, whereas for a nonzero holomorphic one-form $\omega=h^*\eta$,
$
\operatorname{Sing}(\mathcal F_\omega)=Z(\omega)=\operatorname{Crit}(h)\neq\varnothing.
$
Moreover, if $\rho_\omega\colon\pi_1(X)\to\C$ denotes the period homomorphism, then the characters $\gamma\mapsto\exp(t\rho_\omega(\gamma))$, $t\in\C$, generally have infinite image. Such characters cannot be reduced to finite monodromy by passing to a finite \'etale cover. Thus the semisimple condition contains information beyond that coming from finite \'etale covers. It also does not distinguish the compact and dense behavior of the real foliations $\ker\operatorname{Re}(c\omega)$, which is governed by the period group $\Gamma_c$.
\end{remark}

\section*{AI disclosure}
The mathematical content of this paper is the work of the authors. ChatGPT was used to assist with proofreading and to improve the clarity of the exposition.
\section*{Acknowledgements} 
We are grateful to Javier Fern\'andez de Bobadilla, Feng Hao, Claudio Llosa--Isenrich, Simone Noja, Simon Pietig,  Pierre Py,  and Yang Su for helpful conversations and references.


MC is partially supported by the Universit\`a degli Studi di Bari Aldo Moro and is a member of INdAM-GNSAGA.
Partial  financial support  to JK   was provided  by the Simons Foundation   (grant number SFI-MPS-MOV-00006719-02).
SS received funding from the European Research Council (ERC) under the European Union’s Horizon 2020 research and innovation programme under grant agreement No.~948066 (ERC-StG RationAlgic).
The research was partly conducted in the framework of the DFG-funded research training group RTG 2965: From Geometry to Numbers, Project number 512730679. 

\end{document}